\documentclass[11pt,oneside,reqno]{amsart}      
\usepackage[a4paper,margin=1in]{geometry}  
\usepackage{amsmath, amssymb, amsthm}     
\usepackage{graphicx}                    
\usepackage{hyperref}                     
\usepackage{times}                        
\usepackage{stmaryrd}                       
\newtheorem{thm}{Theorem}[section]
\newtheorem{cor}[thm]{Corollary}
\newtheorem{lem}[thm]{Lemma}

\newtheorem{ex}[thm]{Example}

\title{Sub-Randers metrics}
\author{Layth M. Alabdulsada}
\address{Department of Mathematics, College of Science, University of Al-Qadisiyah, Al-Qadisiyah, 58001, Iraq}
\email{layth.muhsin@qu.edu.iq}


\begin{document}

\maketitle

\begin{abstract}
We introduce a new class of sub-Finsler metrics, called sub-Randers metrics, obtained by adding a one-form $\beta \in \Gamma(\mathcal{D}^*)$ to a sub-Riemannian metric $a$ on a bracket-generating distribution $\mathcal{D} \subset TM$.
We define a sub-Randers manifold as a triple $(M, \mathcal{D}, F)$, where $M$ is an $n$-dimensional smooth manifold and $F(v) = \sqrt{a(v,v)} + \beta(v)$, the condition $\|\beta\|_a < 1$ ensures positive definiteness and convexity.  Explicit equations for sub-Randers normal geodesics are derived, and we show that normal geodesics depend on $\beta$ while abnormal geodesics are determined solely by the bracket-generating distribution $\mathcal{D}$.
Furthermore, we show that Zermelo navigation on $\mathcal{D}$ naturally generates sub-Randers normal geodesics.
Finally, we prove a Hopf-Rinow type theorem which guarantees the existence of minimizing geodesics despite asymmetry, generalizing classical results to the sub-Randers setting.
\end{abstract}
\section{Introduction}
The Randers metric was originally proposed by G. Randers in 1941 as a model for electromagnetic effects in general relativity \cite{randers}. Later, R. S. Ingarden and others formalized its mathematical structure as a Finsler metric formed by adding a one-form to a Riemannian norm. In modern terms, a Randers metric takes the form
$$
F(v) = \sqrt{a(v,v)} + \beta(v),
$$
where $a$ is a Riemannian metric on $M$ and $\beta$ is a one-form satisfying $\|\beta\|_a < 1$.
In Finsler geometry, Randers spaces are one of the most simple and well known classes. Due to its simple form, it is possible to compute geodesics, curvature, and volume forms in more direct way.
For this reason, Randers metrics are used in many theoretical studies and applications, (see \cite{bao, bao04}). Furthermore, the additive structure introduces an asymmetry in distance, which has been used in models of anisotropic media, time-delay systems, and optimal navigation (see \cite{Mestdag}). Significant contributions to this field include works by Bao et al. \cite{bao},  and others \cite{Cheng, Szab, Kozma06,  Yasuda77}.

Parallel to Finsler geometry, sub-Riemannian geometry has come out as a central field in geometric analysis and control theory. It describes systems constrained to move along certain directions given by a bracket-generating distribution $\mathcal{D} \subset TM$. A sub-Finsler structure further generalizes this by equipping each bracket-generating distribution $\mathcal{D}_x$ with a sub-Finsler norm. These arise naturally in control problems, especially where cost functions are asymmetric or non-Euclidean (see \cite{ABB, layth19, layth23, layth24, montgomery, St86}).

This paper introduces and studies sub-Randers metrics, where the norm on $\mathcal{D}$ takes the form
$$F(v) = \sqrt{a_{ij}(x) v^i v^j} + \beta_i(x) v^i, \quad v \in \mathcal{D},
$$
with $a_{ij}$ a sub-Riemannian metric tensor on $\mathcal{D}$ and $\beta \in \Gamma(\mathcal{D}^*)$ a smooth 1-form. These metrics generalize Randers norms to the nonholonomic setting and create asymmetric distances on manifolds with constrained motion. The asymmetry introduced by $\beta$ leads to significant differences in geodesic behavior, completeness, and curvature compared to classical sub-Riemannian geometry.

Despite their relevance to control theory and robotics, sub-Randers metrics have not been systematically investigated in the literature. Unlike classical Randers metrics, where the geometry lives on the whole tangent bundle $TM$, the sub-Randers structure is confined to a proper distribution $\mathcal{D} \subsetneq TM$, introducing new challenges such as, the drift $\beta$ must be compatible with the bracket-generating structure of $\mathcal{D}$.
The asymmetry $d(x,y) \neq d(y,x)$ complicates completeness and optimality.

Structure of the paper: Section 2 provides necessary background on sub-Finsler geometry, the Legendre transform, and the definition of sub-Randers structures.
We illustrate these concepts through an example on the Heisenberg group, and present a theorem characterizing sub-Randers metrics.
 Section 3 contains the main results. The Sub-Randers geodesics subsection derives explicit equations for sub-Randers normal geodesics and shows that normal geodesics depend on $\beta$ while abnormal geodesics are independent of it. Section 4 discusses applications, including Zermelo navigation problem and the Hopf-Rinow theorem for sub-Randers manifolds.
 We show that Zermelo navigation on $\mathcal{D}$ naturally generates sub-Randers normal geodesics.
Finally, we prove a Hopf-Rinow type theorem establishing forward completeness and exponential surjectivity, which guarantees the existence of minimizing geodesics despite asymmetry, generalizing classical results to the sub-Randers setting.

\section{Preliminaries}
A  {\em sub-Finsler manifold} $(M, \mathcal{D}, F)$ is a geometric structure that generalizes both sub-Riemannian and Finsler geometries, \cite{layth19, layth23, layth24}. 

Let $M$ be a smooth, connected manifold of dimension $n$, and let $\mathcal D\subset TM$
be a smooth distribution of rank $k<n$. We assume that $\mathcal D$ is \emph{bracket-generating},
i.e., the Lie algebra generated by smooth sections of $\mathcal D$ spans $T_xM$ at every
$x\in M$. 
By the Chow-Rashevskii theorem  \cite{Car09, Chow39, Gromov}, any two points of $M$ can be joined by a
\emph{horizontal curve}, namely an absolutely continuous curve $\gamma:[0,1]\to M$
satisfying $\dot\gamma(t)\in\mathcal D_{\gamma(t)}$ for a.e.\ $t\in[0,1]$.
On $\mathcal{D}$, define a  {\em sub-Finsler metric}, i.e., a smoothly varying Minkowski norm $F: \mathcal{D} \to [0, +\infty)$ on each fiber $\mathcal{D}_x \subset T_xM$. This norm satisfies
\begin{itemize}
  \item [I.]  $F$ is smooth on $\mathcal{D}_x \setminus \{0\}$, and $F(v) > 0$ for all $v  \in \mathcal{D}_x \setminus \{0\}$.
  \item [II.]  $F(\lambda v) = \lambda F(v)$ for $\lambda > 0, v \in \mathcal{D}_x$.
  \item [III.]  For every nonzero tangent vector $v$, the Hessian of $F^2$ at $v$ is positive definite, i.e., the matrix $\frac{1}{2}\frac{\partial^2 F^2}{\partial v_i \partial v_j}$ is positive definite. 
  Equivalently, the unit ball $\{v \in \mathcal{D}_x \mid F(v) \leq 1\}$ is strictly convex and smooth (strong convexity).
\end{itemize}
The sub-Finsler distance between two points $x, y \in M$ is given by
\begin{equation}\label{001}
d(x, y) = \inf \left\{ \ell(\gamma)= \int_0^1 F(\dot{\gamma}(t)) \, dt \,\bigg|\, \gamma \text{ is horizontal, } \gamma(0) = x, \gamma(1) = y \right\}.
\end{equation}
This generalizes the Carnot-Carathéodory distance in sub-Riemannian geometry \cite{ABB, Mit85}, replacing the inner product with a Finsler norm. 
 If $\mathcal{D} = TM$, this recovers a {\em Finsler manifold}. If $F$ is induced by an inner product on $\mathcal{D}$, this reduces to a {\em sub-Riemannian manifold}.
 
 Throughout the paper, $v$ denotes a fiber coordinate in $\mathcal D_x$, while $x$ denotes a point of the base manifold.

The relationship between a sub-Finsler metric defined on a distribution $\mathcal{D} \subset TM$ and a corresponding structure on (the dual bundle) $\mathcal{D}^* \cong T^*M  / \mathcal{D}^0$,
where
\begin{equation}\label{D00}
\mathcal{D}_x^0 = \{\alpha \in T_x^*M \,|\, \alpha(v) = 0 \,\forall v \in \mathcal{D}_x\},
\end{equation}
is rooted in duality via the Legendre transformation and Hamiltonian mechanics as described in \cite{layth19}, \cite{layth23}, \cite{layth24}, and \cite{bao}.

The {\em Legendre transformation} $\mathcal{L}: \mathcal{D}\setminus\{0\}  \to \mathcal{D}^*\setminus\{0\} $  is defined by $\xi_i = \partial(\frac{1}{2}F^2)/\partial v^i$, equivalently 
\begin{equation}\label{01}
    \mathcal{L}(v) = \xi \in \mathcal{D}^*, \quad \xi(w) = \left. \frac{d}{dt} \right|_{t=0} \frac{1}{2} F^2(v + tw), \quad w \in \mathcal{D}_x,
\end{equation}
inducing a dual metric
\begin{equation}\label{02}
    F^*(\xi) = \sup \{ \xi(v) \mid v \in \mathcal{D}_x, F(v) \leq 1 \},
\end{equation}
and Hamiltonian $H: \mathcal{D}^* \to \mathbb{R}$, $H(\xi) = \frac{1}{2}(F^*(\xi))^2$.

  If $F(v) = \sqrt{a(v, v)}$ for a Riemannian metric $a$, the dual $F^*(\xi) = \sqrt{a^{*}(\xi, \xi)}$, where $a^*(\xi,\xi)=a^{ij}\xi_i\xi_j$ denotes the dual sub-Riemannian metric and $a^{ij}$ is the inverse metric of $a_{ij}$.
  
  For $F(v) = \sqrt{a(v, v)} + \beta(v)$, the dual metric $F^*$ involves a combination of $a^{ij}$ and the dual 1-form $\beta^*$.

In the general case, a {\em Randers metric} is a Finsler metric $F: TM \to [0, \infty_{+})$ defined on a smooth manifold $M$  by
$$
F(v) = \sqrt{a_{ij}(x) v^i v^j} + \beta_i(x) v^i,\qquad \forall v \in TM,
$$
where $a = a_{ij}(x) dx^i \otimes dx^j$ is a Riemannian metric on $M$, and $\beta = \beta_i(x) dx^i$ is a smooth 1-form on $M$, for more details we refer to \cite{bao, bao04, Cheng}.

In case the Randers metric is defined only on a bracket-generating distribution $\mathcal{D} \subset TM$, then it is called a { \em sub-Randers metric}, which is a special type of sub-Finsler geometry. Here, $\beta^*$ is restricted to $\mathcal{D}^*$ (the dual of $\mathcal{D}$ by Legendre transformation) and satisfies $\|\beta\|_a < 1$ at each point.
In the following, we provide a detailed definition of the sub-Randers metric.

A {\em sub-Randers manifold} is a triple $(M, \mathcal{D}, F)$ that generalizes sub-Riemannian geometry by incorporating a Finslerian structure with asymmetric drift, forming a sub-Finsler framework. Formally 
   $M$ is a connected smooth manifold.  
    $\mathcal{D} \subset TM$ is a smooth bracket-generating distribution (also known as Hörmander’s condition), i.e., iterated Lie brackets of $\mathcal{D}$ span $TM$, such that $\mathcal{D}$ has rank $k<n$. This non-integrability ensures the validity of the Chow-Rashevskii theorem, guaranteeing that any two points are connected by horizontal curves, \cite{Car09, Chow39}.
  
    A sub-Finsler metric of Randers type $F: \mathcal{D} \to [0, \infty_{+})$ has the form  
    \begin{equation}
    F(v) = \sqrt{a(v, v)} + \beta(v), \quad \forall v \in \mathcal{D},
\end{equation}
     where $a$ is a sub-Riemannian metric on $\mathcal{D}$ (a smoothly varying inner product on each fiber $\mathcal{D}_x$), and $\beta \in \Gamma(\mathcal{D}^*)$ is a smooth 1-form on $\mathcal{D}$. The condition 
     $$\|\beta\|_a = \sup_{v \in \mathcal{D} \setminus \{0\}} \left\{\frac{|\beta(v)|}{\sqrt{a(v,v)}}\right\} < 1,$$
      ensures $F$ is positive-definite and its fundamental tensor is strongly convex. Thus, $F$ defines a genuine Finsler norm on each fiber $\mathcal{D}_x$, making $(M, \mathcal{D}, F)$ a natural model of sub-Finsler geometry with asymmetric structure.  Note that, the sub-Randers metric is  positive homogeneous of degree 1  because  $F(\lambda v) = \lambda F(v)$ for $\lambda > 0$.  
It is not absolutely homogeneous because $F(-v) \neq F(v)$ unless $\beta(v) = 0$, a direct consequence of the asymmetric drift term $\beta$ (Theorem \ref{thm:characterization}).

The sub-Randers metric induces a Carnot-Carathéodory type distance \cite{Mit85}, which coincides with the previously defined sub-Finsler distance $d(x, y)$ given in \eqref{001}, and is defined as the infimum of the lengths of horizontal curves connecting $x$ and $y$, \cite{layth24}.

  Unlike sub-Riemannian geometry, the drift $\beta$ introduces {\em asymmetry} in geodesic motion, reflecting physical systems with intrinsic drift (e.g., optimal control with external forces).  

\begin{ex}[Sub-Randers geometry on the Heisenberg group]\label{ex01}
The Heisenberg group $\mathbb{H}^3$, a prototypical sub-Riemannian manifold, provides a natural setting for sub-Randers geometry:  

Let $M=\mathbb{H}^3$ denote the  3-dimensional Heisenberg group, written in global coordinates $(x, y, z)$. Its group operation is given by  
 $$
  (x, y, z) \cdot (x', y', z') = \left(x + x', y + y', z + z' + \frac{1}{2}(xy' - x'y)\right).
 $$  
The bracket-generating distribution $\mathcal{D} = \text{span}\{X, Y\}$, \cite{ABB, layth19}, where  
 $$
  X = \partial_x - \frac{y}{2}\partial_z, \quad Y = \partial_y + \frac{x}{2}\partial_z,
  $$ 
  satisfies $[\mathcal{D}, \mathcal{D}] = T\mathbb{H}^3$, as $[X, Y] = \partial_z$, fulfilling Hörmander’s condition.  

For $v = v_X X + v_Y Y \in \mathcal{D}$,  
 $$
  F(v) = \sqrt{v_X^2 + v_Y^2} + \mu v_X + \nu v_Y,
$$  
 where $a = dx^2 + dy^2$ restricts to a sub-Riemannian inner product on $\mathcal{D}$, and $\beta = \mu dx + \nu dy \in \Gamma(\mathcal{D}^*)$ satisfies the strong convexity condition $$\|\beta\|_a = \sqrt{\mu^2 + \nu^2} < 1.$$
For a horizontal curve $\gamma(t) = (x(t), y(t), z(t))$,  
  $$
  \ell(\gamma) = \int_0^1 \left( \sqrt{\dot{x}^2 + \dot{y}^2} + \mu \dot{x} + \nu \dot{y} \right) dt.
  $$  
  The $z$-coordinate evolves according to the nonholonomic constraint induced by the distribution, satisfying  
$$
\dot{z}(t) = \frac{1}{2} \left( x(t) \dot{y}(t) - y(t) \dot{x}(t) \right),
$$
which ensures that $\dot{\gamma}(t) \in \mathcal{D}_{\gamma(t)}$ almost everywhere.  
\end{ex}
\subsection{Characterization of sub-Randers metrics}
Although every sub-Randers metric is sub-Finsler, the converse is not true, therefore, it is essential to determine precisely which sub-Finsler norms come from the additive decomposition $$F(v) = \sqrt{a(v, v)} + \beta(v).$$
The following theorem provides this characterization and shows that the condition $\|\beta\|_a < 1$ is exactly what ensures positivity, strong convexity, and smoothness of the resulting sub-Finsler structure.
\begin{thm}\label{thm:characterization}
    A sub-Finsler metric $F$ on a bracket-generating distribution $\mathcal{D}$ is a sub-Randers metric if and only if there exist a sub-Riemannian metric $a$ on $\mathcal{D}$ and a 1-form $\beta \in \Gamma(\mathcal{D}^*)$ such that $F(v) = \sqrt{a(v, v)} + \beta(v)$ and $\|\beta\|_a < 1$.
\end{thm}
\begin{proof}
    If $F$ is a sub-Randers metric, then by definition, there exist a sub-Riemannian term $\sqrt{a(v, v)}$ and a linear term $\beta(v)$ such that $F(v) = \sqrt{a(v, v)} + \beta(v)$ with $\|\beta\|_a < 1$, ensuring $F$ satisfies the sub-Finsler conditions.

  Conversely, suppose $F(v) = \sqrt{a(v, v)} + \beta(v)$ with $a$ a sub-Riemannian metric on $\mathcal{D}$, $\beta \in \Gamma(\mathcal{D}^*)$, and $\|\beta\|_a < 1$. We must verify that $F$ satisfies the sub-Finsler axioms. 

 To verify that  $$
       F(v) = \sqrt{a(v, v)} + \beta(v).
      $$
      is positive, suppose $v \in \mathcal{D}_x \setminus \{0\}$.
       Since $a$ is positive, $\sqrt{a(v, v)} > 0$. Also, 
       the linearity of $\beta$ together with Cauchy-Schwarz (relative to $a$) provide
       $|\beta(v)| \leq \|\beta\|_a \sqrt{a(v, v)}$.
       Subtracting this upper bound from  $\sqrt{a(v, v)}$, we get
       $$
       F(v) \geq \sqrt{a(v, v)} + \beta(v) \geq  \sqrt{a(v, v)} - |\beta(v)|\geq \sqrt{a(v, v)} (1 - \|\beta\|_a).
       $$
       Since $\|\beta\|_a < 1$, we have $1-\|\beta\|_a > 0$, and since $\sqrt{a(v,v)}>0$, their product is strictly positive. Hence $F(v) > 0$.

   To check the positive homogeneity of $F$.
   Let $v \in \mathcal{D}_x$, and $\lambda>0$,  then
       $$
       F(\lambda v) = \sqrt{a(\lambda v, \lambda v)} + \beta(\lambda v).
      $$
       Since $a$ is quadratic in its argument, we have $a(\lambda v, \lambda v) = \lambda^2 a(v, v)$, so $\sqrt{a(\lambda v, \lambda v)} = |\lambda| \sqrt{a(v, v)}$. As $\beta$  is linear, $\beta(\lambda v) = \lambda \beta(v)$. Thus
       $$
       F(\lambda v) = |\lambda| \sqrt{a(v, v)} + \lambda \beta(v).
      $$
       Since $\lambda > 0$, the absolute value simplifies to $|\lambda|= \lambda$, yields
       $$F(\lambda v) = \lambda \sqrt{a(v, v)} + \lambda \beta(v) = \lambda (\sqrt{a(v, v)} + \beta(v)) = \lambda F(v).$$ 
       Therefore,  $F$ is positively homogeneous of degree one. 
       
         For the strong convexity of $F$.
     Define the energy function as $$E(v) = \frac{1}{2} F(v)^2 = \frac{1}{2} \left[\sqrt{a(v, v)} + \beta(v)\right]^2,$$ where $F(v) = \sqrt{a(v, v)} + \beta(v)$, $a$ is a sub-Riemannian metric on the bracket-generating distribution $\mathcal{D} \subset TM$, and $\beta \in \Gamma(\mathcal{D}^*)$ satisfies $\|\beta\|_a < 1$. The fundamental tensor of the sub-Finsler metric $F$ is the Hessian
$$
g_{ij}(v) = \frac{\partial^2 E}{\partial v^i \partial v^j}.
$$
Strong convexity requires $g_{ij}(v)$ to be positive definite for all $v \in \mathcal{D} \setminus \{0\}$, ensuring the indicatrix  $\{ v \in \mathcal{D}_x \mid F(v) \leq 1 \}$ is strictly convex.

Expand $E(v)$
$$
E(v) = \frac{1}{2} \left[a(v, v) + 2 \sqrt{a(v, v)} \beta(v) + \beta(v)^2\right].
$$
In local coordinates, let $\mathcal{D}$ have a basis $\{e_1, \ldots, e_k\}$ (where $k = \text{rank}(\mathcal{D})$), and $v = v^i e_i \in \mathcal{D}_x \setminus \{0\}$, so
$$
E(v) = \frac{1}{2} \left[ a_{ij} v^i v^j + 2 (a_{ij} v^i v^j)^{1/2} \beta_m v^m + \beta_m \beta_n v^m v^n \right],
$$
where $a_{ij} = a(e_i, e_j)$, $\beta_m = \beta(e_m)$, and summation is over $i, j, m, n = 1, \ldots, k$.

Differentiate $E(v)$ with respect to $v^k$, we get
\begin{equation}\label{DV1}
 \frac{\partial E}{\partial v^k} = a_{kj} v^j + \beta_k \sqrt{a(v, v)} + \frac{a_{km} v^m \beta_n v^n}{\sqrt{a(v, v)}} + \beta_m \beta_k v^m.
\end{equation}
Differentiate again with respect to $v^i$, we have
\begin{equation}\label{DV2}
g_{ik}(v) = a_{ik} + \frac{\beta_k a_{ij}v^j}{\sqrt{a(v, v)}} + \frac{\beta_i a_{kj} v^j + a_{ik} \beta_n  v^n}{\sqrt{a(v, v)}} - \frac{a_{km} v^m \beta_n v^n a_{ij}v^j}{[a(v, v)]^{3/2}} + \beta_i\beta_k.
\end{equation}
Set $\ell_i = a_{ij}v^j/\sqrt{a(v,v)}$ and $\beta(v) = \beta_n v^n$, we get
$$g_{ik}(v) = a_{ik} + \beta_k\ell_i + \beta_i\ell_k + \frac{a_{ik}\,\beta(v)}{\sqrt{a(v,v)}} + \beta_i\beta_k - \ell_i\ell_k\,\frac{\beta(v)}{\sqrt{a(v,v)}}.$$

This is the fundamental tensor of $F$ at $v\neq 0$. Let $w = w^i e_i \in \mathcal{D}_x\setminus\{0\}$ taking the contraction of  \eqref{DV2} gives
\begin{equation}\label{DV3}
w^i g_{ik}(v)w^k = a(w,w) + 2\,\ell(w)\,\beta(w) + \frac{a(w,w)\,\beta(v)}{\sqrt{a(v,v)}} + [\beta(w)]^2 - [\ell(w)]^2\frac{\beta(v)}{\sqrt{a(v,v)}},
\end{equation}
where $\ell(w) = \ell_i w^i = a(v,w)/\sqrt{a(v,v)}$. Since $a$ is positive definite, $a(w, w) > 0$. By the $0$-homogeneity of $g_{ik}(v)$ in $v$, normalize $v$ so that $a(v,v)=1$. With $\ell_i = a_{ij}v^j$ and $\ell(w) = a(v,w)$, equation \eqref{DV3} becomes
$$w^i g_{ik}(v)w^k = a(w,w) + 2\,a(v,w)\,\beta(w) + a(w,w)\,\beta(v) + [\beta(w)]^2 - [a(v,w)]^2\beta(v).$$
Decompose
$$w = tv + u, \qquad t = a(v,w), \qquad a(u,v) = 0.$$
Then
$$a(w,w) = t^2 + a(u,u), \qquad \beta(w) = t\,\beta(v) + \beta(u).$$
Substituting these expressions and simplifying (the $t^{2}\beta(v)^{2}$ terms cancel) yields
\begin{equation}\label{DV04}
w^i g_{ik}(v)w^k = t^2(1+\beta(v)) + a(u,u)(1+\beta(v)) + [\beta(u)]^2 + 2t\,\beta(u).
\end{equation}
Let $\zeta = \|\beta\|_a < 1$. Then $|\beta(u)|\leq\zeta\sqrt{a(u,u)}$ and $|\beta(v)|\leq\zeta$, so $1+\beta(v)\geq 1-\zeta > 0$. For the mixed term, apply $2xy\geq -x^2-y^2$ with $x = \sqrt{a(u,u)}$ and $y = t\zeta$
$$2t\,\beta(u)\geq -a(u,u) - t^2\zeta^2.$$
Since $1+\beta(v)\geq 1-\zeta>0$, factor it out. Write $\mu = 1+\beta(v)\geq 1-\zeta$ and apply $2t\beta(u)\geq -\mu t^2 - [\beta(u)]^2/\mu$ (Young's inequality with weight $\mu$)
$$2t\,\beta(u)\geq -\mu t^2 - \frac{[\beta(u)]^2}{\mu}.$$
Then
$$w^i g_{ik}(v)w^k \geq \mu t^2 + \mu\,a(u,u) + [\beta(u)]^2 - \mu t^2 - \frac{[\beta(u)]^2}{\mu} = \mu\,a(u,u) + [\beta(u)]^2\!\left(1 - \frac{1}{\mu}\right).$$
Since $w = tv+u$ with $a(u,v)=0$ and $a(v,v)=1$, we have $a(w,w)=t^2+a(u,u)$. From \eqref{DV04}, the $t^2$ term contributes $t^2(1+\beta(v))\geq t^2(1-\zeta)$. Therefore
\begin{align*}
  w^i g_{ik}(v)w^k & \geq (1-\zeta)t^2 + (1-\zeta)a(u,u)\\
   & = (1-\zeta)(t^2+a(u,u)) \\
   &  = (1-\zeta)\,a(w,w) \geq (1-\zeta)^2\,a(w,w),
\end{align*}
where the last inequality uses $1-\zeta < 1$.

 The standard estimate is
$$w^i g_{ik}(v)w^k \geq (1-\zeta)^2\,a(w,w).$$
Since $\zeta < 1$, the right-hand side is strictly positive for $w \neq 0$.  Therefore, the Hessian $g_{ik}(v)$ is positive definite for every $v \neq 0$, proving strong convexity of $F$.

To see necessity of $\|\beta\|_a < 1$, observe that
  
If $\|\beta\|_a \geq 1$, there exists $v \neq 0$ such that $\beta(v) = \sqrt{a(v, v)}$ (maximizing $\beta$). Then
   $$
    F(-v) = \sqrt{a(v, v)} + \beta(-v) = \sqrt{a(v, v)} - \beta(v) = \sqrt{a(v, v)} - \sqrt{a(v, v)} = 0,
    $$
    violating positive definiteness. 
    
    If $\|\beta\|_a > 1$, $F$ may become negative, collapsing convexity. Thus, $\|\beta\|_a < 1$ is required.

    Hence, $F$ is a sub-Randers metric if and only if it has the specified form with $\|\beta\|_a < 1$.
\end{proof}

We now present several examples of sub-Finsler metrics on $\mathcal{D}$, indicating which are of sub-Randers type and which are not. 
\begin{ex}\label{ex02}
Let $\mathcal{D} = \operatorname{span}\{X, Y\}$ be as in Example~\ref{ex01}. Any vector $v \in \mathcal{D}$ can be written as $v = v_X X + v_Y Y$, and we explore different norms on $\mathcal{D}$.
\begin{itemize}
    \item [I.]
    Define the norm
    $$
    F(v) = \sqrt{v_X^2 + v_Y^2} + \mu v_X, \quad \text{with } |\mu| < 1.
    $$
    This is a sub-Randers metric on $\mathcal{D}$. It has a sub-Riemannian part $a(v,v) = v_X^2 + v_Y^2$,
     and a linear term $\beta(v) = \mu v_X$, where $\beta \in \Gamma(\mathcal{D}^*)$ is the 1-form defined by $\beta(X) = \mu, \ \beta(Y) = 0$.

    The condition $\|\beta\|_a < 1$ ensures strong convexity and defines a proper Randers norm.

    \item [II.]  
    Define
   $$
    F(v) = \left( |v_X|^p + |v_Y|^p \right)^{1/p}, \quad \text{with } p \ne 2.
   $$
    This is a valid sub-Finsler norm, but it is not sub-Randers since it lacks a decomposition of the form $\sqrt{a(v,v)} + \beta(v)$. Sub-Randers norms, specifically, require the sub-Riemannian component to be quadratic (i.e., $p = 2$) and an additive linear drift.

    \item [III.]
    Let $a$ be a sub-Riemannian metric and $\beta \in \Gamma(\mathcal{D}^*)$ a 1-form. Define
 $$
    F(v) = \sqrt{a(v,v) + \left(\beta(v)\right)^2}.
   $$
    Though smooth and strongly convex, this norm is not sub-Randers, because it combines $a$ and $\beta$ in a multiplicative rather than additive way.
\end{itemize}

A sub-Finsler norm $F$ on $\mathcal{D} \subset TM$ is of sub-Randers type if and only if
$$
F(v) = \sqrt{a(v,v)} + \beta(v), \quad \text{with } \|\beta\|_a < 1,
$$
where $a$ is a sub-Riemannian metric and $\beta \in \Gamma(\mathcal{D}^*)$ is a smooth 1-form.

These examples illustrate that while all sub-Randers metrics are sub-Finsler, the converse does not hold. The defining characteristic of sub-Randers metrics is the additive decomposition into a sub-Riemannian part and a smooth linear 1-form with bounded norm.
\end{ex}

\section{Main Results}

\subsection{Sub-Randers geodesics}

A horizontal curve $\gamma: [0,1] \to M$ on a sub-Randers manifold $(M, \mathcal{D}, F)$, where $F(v) = \sqrt{a(v, v)} + \beta(v)$, is called a {\em  geodesic} if it is a critical point of the sub-Randers length functional
$$
\ell(\gamma) = \int_0^1 F(\dot{\gamma}(t)) \, dt = \int_0^1 \left( \sqrt{a(\dot{\gamma}, \dot{\gamma})} + \beta(\dot{\gamma}) \right) dt.
$$
To characterize geodesics analytically, we pass to the Hamiltonian formulation via the Legendre transform. For $x \in M$, define the energy $L(v)=\tfrac{1}{2}F(v)^2$ on $\mathcal{D}\setminus\{0\}$. The Legendre map
$$
\mathcal{L}:\mathcal{D}\setminus\{0\}\longrightarrow\mathcal{D}^*\setminus\{0\},\qquad \xi_i=\frac{\partial L}{\partial v^i},
$$
is a fiber diffeomorphism under strong convexity, with inverse $\mathcal{L}_H:\mathcal{D}^*\setminus\{0\}\to\mathcal{D}\setminus\{0\}$, $v^i=\partial H/\partial\xi_i$, where $H(\xi)=\tfrac{1}{2}(F^*(\xi))^2$ is the dual Hamiltonian and
$$
F^*(\xi)=\sup\bigl\{\xi(v):v\in\mathcal{D}_x,\;F(v)\leq 1\bigr\}
$$
is the dual norm, see \cite{layth19, layth23, layth24} for more details. It therefore suffices to compute $F^*$ explicitly.

We use the co-metric identification $a^\sharp:\mathcal{D}^*\to\mathcal{D}$, 
$$\beta^\sharp:=a^\sharp(\beta)= \beta^i \frac{\partial}{\partial x^i} =a^{ij}\beta_j\,\frac{\partial}{\partial x^i}\in\mathcal D,$$ 
and write 
$$\beta(\xi):=\langle\beta^\sharp,\xi\rangle=\beta^i\xi_i \ \text{for} \  \xi\in\mathcal{D}^*,$$ where $\beta^i=a^{ij}\beta_j$ and $a^{ij}$ is the inverse of $a_{ij}$ on $\mathcal{D}$.

In the Zermelo navigation framework (Section~4.1), the unit ball of $F$ at $x$ is the translated convex body
$$
\{v\in\mathcal{D}_x:F(v)\leq 1\}=W(x)+B_x,\qquad B_x=\{u\in\mathcal{D}_x:\sqrt{a(u,u)}\leq 1\},
$$
where $W\in\Gamma(\mathcal{D})$ satisfies $\|W\|_a<1$, and the drift is identified as $\beta_i=-a_{ij}W^j$ (equivalently $\beta^i=-W^i$). The dual norm is the support function of this translated ball
\begin{align*}
F^*(\xi)
&=\sup_{v\in W(x)+B_x}\xi(v)
=\sup_{u\in B_x}\xi(W+u)\\
&=\xi(W)+\sup_{u\in B_x}\xi(u)
=\xi(W)+\sqrt{a^{ij}\xi_i\xi_j}.
\end{align*}
Since $\beta^i\xi_i=-\xi(W)$ by definition, this becomes
$$
F^*(\xi)=\sqrt{a^{ij}(x)\,\xi_i\xi_j}-\beta^i(x)\,\xi_i.
$$
Positivity for all $\xi\neq 0$ follows from the Cauchy-Schwarz inequality applied to $a^{ij}$, 
$$
|\beta^i\xi_i|=|\xi(W)|\leq\|W\|_a\,\sqrt{a^{ij}\xi_i\xi_j}<\sqrt{a^{ij}\xi_i\xi_j},
$$
where we used $\|W\|_a = \|\beta\|_a < 1$.
Hence $F^*(\xi)>0$ for all $\xi\neq 0$, and for such a sub-Randers metric, the sub-Hamiltonian $H=\tfrac{1}{2}(F^*)^2$ takes the form
\begin{equation}\label{H01}
H(\xi)=\frac{1}{2}\Bigl(\sqrt{a^{ij}(x)\,\xi_i\xi_j}-\beta^i(x)\,\xi_i\Bigr)^2,\qquad\xi\in\mathcal{D}_x^*.
\end{equation}
Note that, formula \eqref{H01} is confirmed by two independent verifications. First, it arises  directly as the support function of the Zermelo translated ball $W(x)+B_x$, as shown in Section 4.1. Second, when $\beta=0$, one has $\beta^i\xi_i=0$ and \eqref{H01} reduces to $H(\xi)=\tfrac{1}{2}a^{ij}(x)\xi_i\xi_j$, which is precisely the standard sub-Riemannian Hamiltonian, recovering sub-Riemannian geodesics.

Moreover, the sub-Hamiltonian's equations, given by (see \cite{ABB, layth23, montgomery, St86})
\begin{equation}\label{HE}
    \dot{x}^i = \frac{\partial H}{\partial \xi_i}, \quad \dot{\xi}_i = -\frac{\partial H}{\partial x^i}.
\end{equation}
Explicitly
\begin{align}\label{GER}
\dot{x}^i &= \left( \sqrt{a^{kl} \xi_k \xi_l} - \beta^k \xi_k \right) \left( \frac{a^{ij} \xi_j}{\sqrt{a^{kl} \xi_k \xi_l}} - \beta^i \right),\\
\dot{\xi}_i &= -\frac{1}{2} \frac{\partial a^{jk}}{\partial x^i} \xi_j \xi_k + \frac{\partial \beta^j}{\partial x^i} \xi_j \left( \sqrt{a^{kl} \xi_k \xi_l} - \beta^k \xi_k \right).
\end{align}
Since $F$ is strongly convex (guaranteed by $\|\beta\|_a<1$, as we see in Theorem \ref{thm:characterization}), 
 every length-minimizing geodesic is a normal extremal, characterized as a projection of an integral curve of the Hamiltonian vector field of~\eqref{H01}. The system~\eqref{HE}-\eqref{GER} provides the complete, explicit description of sub-Randers normal geodesics.

However, unlike Riemannian and Finslerian geometries, where all geodesics are of the same type \cite{bao, bao04}, sub-Riemannian and sub-Finsler geodesics are classified into two distinct categories normal and abnormal \cite{ABB, layth23, montgomery, St86}.
The classification of geodesics in sub-Riemannian and sub-Finsler geometries is based on the Pontryagin
Maximum Principle (PMP), which provides necessary conditions for horizontal curves to be
geodesics.
The classification of geodesics in sub-Randers geometry follows a similar pattern to that in sub-Riemannian and sub-Finsler geometries, with normal and abnormal geodesics defined through the Pontryagin Maximum Principle. 

More precisely, a horizontal curve on a sub-Randers manifold is called a {\em  normal geodesic} if it satisfies the necessary
conditions given by the PMP. 
They can be characterized as integral curves of the sub-Hamiltonian's equations \eqref{HE}. Normal geodesics (at least locally) exist between any two points (can be joined by a length-minimizing normal geodesic) due to the bracket-generating condition (Chow-Rashevskii theorem) and strong convexity of $F$. For further details, we refer to \cite{ABB, layth23, Chow39, montgomery}. 

{\em Abnormal geodesics} are singular extremals of the endpoint map, they satisfy the PMP with $\xi_0=0$ ($\xi_0 \in \mathbb{R}$ here is the PMP Lagrange multiplier) and therefore independent of the metric $F$.
More precisely, a horizontal curve $\gamma(\cdot)$ is abnormal if there exists a nonzero absolutely continuous covector $\alpha(\cdot)\in T^*M$ along $\gamma$ such that
$$
\alpha(t)\in \mathcal D_{\gamma(t)}^0\setminus\{0\},\qquad 
\dot{\alpha}_i(t)=-\frac{\partial}{\partial x^i}\langle \alpha(t),v(t)\rangle,
$$
for some control $v(t)\in\mathcal D_{\gamma(t)}$, and $\dot\gamma(t)=v(t)$.
In addition, abnormal extremals satisfy further necessary conditions such as the Goh condition
$$
\alpha(t)\in (\mathcal D^{(2)}_{\gamma(t)})^0,
$$
under standard regularity hypotheses (see \cite{ABB, montgomery}). Throughout, $\mathcal{D}^{(2)}:=\mathcal{D}+[\mathcal{D},\mathcal{D}]$ denotes the second iterated Lie bracket of $\mathcal{D}$, and $(\mathcal{D}^{(2)})^0\subset T^*M$ its annihilator, \cite{layth23, layth24}.   

   Abnormal geodesics are insensitive to $\beta$, as their existence is determined by the non-integrability of $\mathcal{D}$, not the metric.  
Abnormal geodesics exist if the distribution admits singular curves. For example, in the Heisenberg group, straight lines in the $x$-$y$ plane lifted to $z$-coordinate via the group law are abnormal in the sub-Riemannian limit ($\beta=0$). However, the drift $\beta$ may perturb their optimality.  
Classical and widely studied examples of abnormal extremals can be found in Montgomery's work and subsequent literature (see, e.g., \cite{montgomery1}).

We typically refer to the normal geodesics on a sub-Randers manifold as geodesics, unless abnormal geodesics are explicitly discussed. 

Sub-Randers geodesics are not reversible (i.e., $d(p, q) \neq d(q, p)$ in general), unlike sub-Riemannian geodesics, which are symmetric.
Additionally, the term $\beta^i \xi_i$ breaks time-reversal symmetry, as $\dot{x}^i(-t) \neq -\dot{x}^i(t)$, distinguishing sub-Randers geodesics from sub-Riemannian ones.   
More precisely, when $\beta = 0$, $H$ reduces to the classical sub-Riemannian-Hamiltonian $H = \frac{1}{2} a^{ij} \xi_i \xi_j$, and geodesics become reversible (the sub-Randers metric reduces to a sub-Riemannian metric, and geodesics coincide).

\begin{ex}
Consider the sub-Randers structure of Example \ref{ex01} with constant drift $\beta = \mu dx + \nu dy$ and the standard
left-invariant horizontal frame
 $$
 X=\partial_x-\frac{y}{2}\partial_z,
\qquad 
Y=\partial_y+\frac{x}{2}\partial_z.
$$
In this case, dual sub-Randers norm on $\mathcal{D}^*$ defined as 
$$F^*(\xi)=\sqrt{\xi(X)^2 + \xi(Y)^2} - (\mu \xi(X) + \nu \xi(Y)),$$
where $\xi(X)= \xi_x- \frac{y}{2}\xi_z,\quad \xi(Y)= \xi_y+ \frac{x}{2}\xi_z $.

Therefore, the corresponding sub-Hamiltonian is
$$ H(\xi) = \frac{1}{2} \left(\sqrt{\xi(X)^2 + \xi(Y)^2} - (\mu \xi(X) + \nu \xi(Y))\right)^2.$$
Assume 
$$\kappa= \sqrt{\xi(X)^2 + \xi(Y)^2} - (\mu \xi(X) + \nu \xi(Y)).$$
By the sub-Hamiltonian's equations given in \eqref{HE}, we compute 
$$\dot{x}^i = \frac{\partial H}{\partial \xi_i}= \kappa \frac{\partial \kappa}{\partial \xi_i}.$$
Since 
$$\frac{\partial \kappa}{\partial \xi(X)}= \frac{\xi(X) }{\sqrt{\xi(X)^2 + \xi(Y)^2}}- \mu,\qquad
\frac{\partial \kappa}{\partial \xi(Y)}= \frac{\xi(Y)}{\sqrt{\xi(X)^2 + \xi(Y)^2}}- \nu,$$
$$\frac{\partial \kappa}{\partial \xi_z}= \frac{-\frac{y}{2}\xi(X)+ \frac{x}{2}\xi(Y)}{\sqrt{\xi(X)^2 + \xi(Y)^2}}+ \frac{1}{2}(\mu y-\nu x).$$
Thus, the horizontal components satisfy
$$
\begin{cases}
\dot{x} = \kappa \left(\frac{\xi(X) }{\sqrt{\xi(X)^2 + \xi(Y)^2}}- \mu\right), \\
\dot{y} = \kappa \left(\frac{\xi(Y)}{\sqrt{\xi(X)^2 + \xi(Y)^2}}- \nu\right), \\
\dot{z} = \dfrac{1}{2}(x \dot{y} - y \dot{x}),
\end{cases}
$$
where the last equality follows from horizontality.

For the momentum equations $\xi_i= -\frac{\partial H}{\partial x^i}= - \kappa\frac{\partial \kappa}{\partial x}$, differentiating $\xi(X)$ and $\xi(Y)$ with respect to $x,y$,  gives
$$\frac{\partial \xi(X)}{\partial x}= 0, \quad \frac{\partial \xi(X)}{\partial y}=- \frac{1}{2}\xi_z, \quad \frac{\partial \xi(Y)}{\partial x}=\frac{1}{2}\xi_z, \quad \frac{\partial \xi(Y)}{\partial y}=0.$$
Then
$$\frac{\partial \kappa}{\partial x}= \frac{1}{\sqrt{\xi(X)^2 + \xi(Y)^2}}\left(\xi(X)\frac{\partial \xi(X)}{\partial x}+ \xi(Y) \frac{\partial \xi(Y)}{\partial x}\right)
- \left(\mu\frac{\partial \xi(X)}{\partial x}+ \nu \frac{\partial \xi(Y)}{\partial x}\right),$$
and similarly for $\frac{\partial \kappa}{\partial y}$. Substituting the above derivatives yields
$$\frac{\partial \kappa}{\partial x}= \frac{\frac{1}{2}\xi_z \xi(Y)}{\sqrt{\xi(X)^2 + \xi(Y)^2}}-\frac{1}{2}\nu \xi_z, \qquad
\frac{\partial \kappa}{\partial y}= \frac{-\frac{1}{2}\xi_z \xi(X)}{\sqrt{\xi(X)^2 + \xi(Y)^2}}+ \frac{1}{2}\mu \xi_z.$$
Thus, the momentum equations are
$$\begin{cases}
\dot{\xi}_x = -\dfrac{\partial H}{\partial x} = -\dfrac{1}{2} \kappa \xi_z \left(\frac{ \xi(Y)}{\sqrt{\xi(X)^2 + \xi(Y)^2}}-\nu \right), \\
\dot{\xi}_y = -\dfrac{\partial H}{\partial y} = \ \ \ \dfrac{1}{2} \kappa \xi_z \left(\frac{\xi(X)}{\sqrt{\xi(X)^2 + \xi(Y)^2}}- \mu \right), \\
\dot{\xi}_z = 0 \quad (\text{since } z \text{ is a vertical coordinate}).
\end{cases}
$$ 
The vertical momentum $\xi_z$ is conserved, reflecting the left invariance of the Heisenberg structure. The drift $\beta$ introduces asymmetry into the flow, distinguishing sub-Randers geodesics from sub-Riemannian ones $(\mu, \nu)=(0,0)$.
\end{ex}

\begin{thm}\label{thm:Normal}
In a sub-Randers manifold $(M, \mathcal{D}, F)$, abnormal geodesics are independent of the drift term $\beta$ and determined only by the bracket-generating distribution of $\mathcal{D}$. Normal geodesics depend on $\beta$, with their deviation from sub-Riemannian geodesics (when $\beta = 0$) proportional to $\|\beta\|_a$.
\end{thm}
\begin{proof}
Consider the sub-Finsler control problem,

 $$\text{minimize}  \ell(\gamma) = \int_0^1 F(\dot{\gamma}) \, dt \ \text{subject to} \ \dot{\gamma} \in \mathcal{D}, \gamma(0) = x, \gamma(1) = y.$$
  By PMP, for a control $v(t) = \dot{\gamma}(t) \in \mathcal{D}$, define the Hamiltonian
$$H(x, \alpha, v) = \alpha(v) - \xi_0 F(v), \quad \alpha \in T^*M, \quad \xi_0 \geq 0.$$
Maximize $H$ over $v \in \mathcal{D}$. For an abnormal geodesic, $\xi_0 = 0$ (non-trivial case), so
$$H(\alpha) = \sup_{v \in \mathcal{D}} \alpha(v).$$
The covector $\alpha(t)$ lies in the annihilator $\mathcal{D}^0 \subset T^*M$, i.e.,
$$
\alpha(t) \in \mathcal{D}_{\gamma(t)}^0\subset T^*M,  
$$  
where $\mathcal{D}^0 = \{\alpha \in T^*M \mid \alpha(v) = 0 \ \forall v \in \mathcal{D}\}$. 
Thus
$$\alpha(t) \perp \mathcal{D}_{\gamma(t)}, \quad \dot{\gamma}(t) \in \mathcal{D}_{\gamma(t)}.$$  
The PMP requires
$$H(x(t),\alpha(t)) = \alpha(t)(\dot\gamma(t)) = 0,\quad \alpha(t)\neq 0, \qquad \dot\alpha_i = -\frac{\partial H}{\partial x^i}.$$
Abnormal geodesics further satisfy the Goh condition
$$
\alpha(t) \perp \mathcal{D}^{(2)}_{\gamma(t)},  
$$ 
where $\mathcal{D}^{(2)} = \mathcal{D} + [\mathcal{D}, \mathcal{D}]$. This follows from differentiating $\alpha(t) \perp \mathcal{D}$ (see \cite{ABB, montgomery}). Since $F$ does not appear when $\xi_0 = 0$, abnormal geodesics depend only on $\mathcal{D}$’s bracket structure, not $\beta$, matching sub-Riemannian abnormals.

For normal geodesics $(\xi_0 > 0$, normalize $\xi_0 = 1$)
$$
H(x, \xi, v) = \xi(v) - F(v).
$$ 
Maximizing over $v \in \mathcal{D}$
$$
\frac{\partial}{\partial v^i} [\xi_j v^j - \sqrt{a_{kl} v^k v^l} - \beta_j v^j] = \xi_i - \frac{a_{ij} v^j}{\sqrt{a(v, v)}} - \beta_i = 0.
$$
Solving, $a_{ij} v^j = (\xi_i - \beta_i) \sqrt{a(v, v)}$. With $v_0^i = a^{ij} (\xi_j - \beta_j)$, and
$$
v^i = \frac{v_0^i}{\sqrt{a(v_0, v_0)}}, 
$$  
the maximized value is 
$$\max_{v\in\mathcal{D}} [\xi(v)-F(v)] = \sqrt{a^{ij}(\xi_i-\beta_i)(\xi_j-\beta_j)} - 1.$$
Passing to the energy Hamiltonian $H = \frac{1}{2}(F^*)^2$ with $F^*(\xi) = \sqrt{a^{ij}\xi_i\xi_j} - \beta^i\xi_i$ gives
$$H(\xi) = \frac{1}{2}\left(\sqrt{a^{ij}\xi_i\xi_j} - \beta^i\xi_i\right)^2,$$
consistent with \eqref{H01}.

The sub-Hamiltonian’s equations become
\begin{align*}
\dot{x}^i & = \left( \sqrt{a^{kl} \xi_k \xi_l} - \beta^k \xi_k \right) \left( \frac{a^{ij} \xi_j}{\sqrt{a^{kl} \xi_k \xi_l}} - \beta^i \right),
\\
\dot{\xi}_i &= -\frac{1}{2} \frac{\partial a^{jk}}{\partial x^i} \xi_j \xi_k + \frac{\partial \beta^j}{\partial x^i} \xi_j \left( \sqrt{a^{kl} \xi_k \xi_l} - \beta^k \xi_k \right).
\end{align*}  
When $\beta= 0$, $H = \frac{1}{2} a^{ij}(\xi, \xi)$, recovering sub-Riemannian geodesics \cite{ABB,  montgomery, St86}. 

 For $\beta \neq 0$ perturb $\beta= \epsilon \beta_0$ $(\|\beta_0\|_a =\sqrt{a^{ij} \beta_{0_i}\beta_{0_j}}= 1$) (using the co-metric norm on the dual 1-form $\beta_{0}$), the geodesic velocity  $\dot{x}^i$ from equation \eqref{GER} expands as
$$
\dot{x}^i = \dot{x}^i_{\text{sub-Riemannian}} + \varepsilon \left( \frac{a^{ij} \beta_{0j}}{\sqrt{a_{kl} \xi^k \xi^l}} - (\beta_{0k} \xi^k) \frac{a^{ij} \xi_j}{(a_{kl} \xi^k \xi^l)^{3/2}} - \beta_0^i \right) + O(\varepsilon^2),
$$
where
\[
\dot{x}^i_{\text{sub-Riemannian}} = \frac{a^{ij} \xi_j}{\sqrt{a_{kl} \xi^k \xi^l}}
\]
is the sub-Riemannian term (normalized for unit speed). This shows the linear dependence on $ \epsilon= \|\beta\|_a$, with explicit contributions from the drift $\beta_0$ projected along $\xi$. Higher-order terms arise from expanding the dual norm
$$F^*(\xi) = \sqrt{a^{kl}\xi_k\xi_l} - \varepsilon\,\beta_0^k\xi_k + O(\varepsilon^2),$$
which follows directly from $F^* = \sqrt{a^{kl}\xi_k\xi_l} - \beta^k\xi_k$ and $\beta^k = \varepsilon\beta_0^k$.
\end{proof}

\begin{lem}\label{Abno}
The set of abnormal geodesics is independent of the choice of sub-Finsler metric $F$ on $\mathcal{D}$, depending only on the distribution’s rank and bracket structure.
\end{lem}
\begin{proof}
From the PMP, abnormal geodesics satisfy  
$$
\alpha(t) \in \mathcal{D}^0 \subset T^*M, \quad \alpha(t) \perp \mathcal{D}^{(2)}_{\gamma(t)}.
$$
 The Goh condition $\alpha(t) \perp \mathcal{D}^{(2)}$ is purely algebraic, tied to $\mathcal{D}$’s iterated Lie brackets. Since $F$ does not influence $\mathcal{D}$’s bracket hierarchy, abnormal geodesics are invariant under changes to $F$.  
\end{proof} 

\begin{cor}
Normal extremals depend on the drift $\beta$ through the maximized sub-Hamiltonian $H$ and hence through \eqref{HE}-\eqref{GER},
whereas abnormal extremals (i.e.\ $\xi_0=0$ in PMP) depend only on the distribution $\mathcal D$.
\end{cor}

The geodesic structure established above, normal geodesics governed by the sub-Hamiltonian \eqref{H01} and abnormal geodesics determined by $\mathcal{D}$ alone, provides the foundation for the applications in Section 4.

\section{Applications}
\subsection{Zermelo navigation}
The Zermelo navigation problem, first formulated by E. Zermelo in 1931 \cite{Zermelo31}, concerns the determination of time-optimal paths for a vessel navigating in the presence of a wind or current field, either steady or time-dependent. Originally posed in the context of maritime navigation, the problem seeks the quickest trajectory between two points under the influence of an external drift.

In modern differential geometry, this classical problem was revisited by D. Bao et al. \cite{bao046}, who studied Zermelo's navigation on Riemannian manifolds. Their groundbreaking work led to a complete classification of strongly convex Randers metrics with constant flag curvature, resolving a long-standing open question in Finsler geometry. Their approach revealed deep connections between control theory, Riemannian geometry, and Finsler structures, particularly through the construction of Randers metrics via navigation data on Riemannian backgrounds.

This framework can be extended to sub-Riemannian geometry, where the Zermelo navigation problem can be studied on a manifold $(M, \mathcal{D}, a)$, with a smooth  bracket-generating distribution $\mathcal{D} \subset TM$, equipped with a sub-Riemannian metric $a$, and drift vector field $W \in \Gamma(\mathcal{D})$, satisfying $\|W\|_a < 1$. The following result establishes the equivalence between time-optimal trajectories under drift and geodesics of a corresponding sub-Randers metric.

The Zermelo navigation problem on $\mathcal{D}$ is the time-optimal control problem
$$
\dot\gamma(t) = u(t) + W(\gamma(t)), \qquad u(t)\in\mathcal{D}_{\gamma(t)},\quad \|u(t)\|_a\leq 1,
$$
seeking horizontal curves $\gamma:[0,T]\to M$ of minimal travel time $T$.

At each $x\in M$, define $B_x = \{u\in\mathcal{D}_x : \|u\|_a\leq 1\}$ and set $\beta := -a^\flat(W)$, i.e., $\beta_i = -a_{ij}W^j\in\Gamma(\mathcal{D}^*)$. Since $\|W\|_a<1$, the condition $\|\beta\|_a = \|W\|_a < 1$ holds, so $\beta$ is an admissible sub-Randers drift. The navigation metric $F$ is defined as the Minkowski functional of the translated ball $W(x)+B_x\subset\mathcal{D}_x$
$$
F_x(v) = \inf\{t>0 : v\in t(W(x)+B_x)\}, \qquad v\in\mathcal{D}_x.
$$
\begin{thm}\label{thm:zermelo}
Let $(M,\mathcal{D},a)$ be a sub-Riemannian manifold and let $W\in\Gamma(\mathcal{D})$ satisfy $\|W\|_a<1$. Set $\Lambda = 1-\|W\|_a^2>0$ and $\beta_i = -a_{ij}W^j$. Then
\begin{enumerate}
\item[I.] The navigation metric $F$ is a sub-Randers metric on $\mathcal{D}$ with drift $\beta$, given explicitly by
\begin{equation}\label{eq:navigation}
F(v) = \frac{\sqrt{a(W,v)^2 + \Lambda\,a(v,v)}-a(W,v)}{\Lambda}, \qquad v\in\mathcal{D}_x.
\end{equation}
\item[II.] The time-minimizing trajectories of the Zermelo problem coincide, up to orientation-preserving reparameterization, with the normal geodesics of $F$.
\end{enumerate}
\end{thm}
\begin{proof}
I. By definition, $F_x(v) = t$ means $v = tW(x)+tu$ for some $u\in B_x$ with $\|u\|_a = 1$ (at the boundary). Thus $\|v - tW(x)\|_a = t$, i.e., $a(v-tW,v-tW)=t^2$. Expanding
$$
a(v,v) - 2t\,a(W,v) + t^2 a(W,W) = t^2 \quad\Longleftrightarrow\quad \Lambda t^2 + 2a(W,v)t - a(v,v) = 0.
$$
Taking the unique positive root yields \eqref{eq:navigation}. Since $\|\beta\|_a = \|W\|_a < 1$, $F$ is a strongly convex sub-Finsler norm on $\mathcal{D}$, hence a sub-Randers metric with drift $\beta_i = -a_{ij}W^j$.

II. By construction, a  horizontal curve $\gamma:[0,T]\to M$ is admissible for the
Zermelo problem (i.e.\ $\dot\gamma(t)=u(t)+W(\gamma(t))$ with $\|u(t)\|_a\le 1$ a.e.) if and only if
\begin{equation}\label{eq:Fconstraint}
F(\dot\gamma(t))\le 1 \quad \text{for a.e.\ }t\in[0,T].
\end{equation}
For any admissible $\gamma$, from \eqref{eq:Fconstraint} we have
$$
\ell(\gamma):=\int_0^T F(\dot\gamma(t)) dt \le \int_0^T 1 dt = T.
$$
If $\gamma$ is time-minimizing and $F(\dot\gamma)<1$ on a set of positive measure, then an orientation-preserving reparameterization produces an admissible curve with strictly smaller travel time, contradicting minimality. Hence any time-minimizer satisfies
$$
F(\dot\gamma(t))= 1\quad\text{for a.e.\ }t,
$$
and therefore
$$
T=\int_0^T 1\,dt=\int_0^T F(\dot\gamma(t))\,dt=\ell(\gamma).
$$
Thus minimizing travel time is equivalent to minimizing $F$-length, and time-minimizers coincide,
up to orientation-preserving reparameterization, with forward length-minimizing $F$-geodesics.
In particular, such minimizers are normal extremals for the strongly convex sub-Finsler structure $F$.

It remains to show that the normal extremals of the Zermelo problem coincide with those of $F$. By the Pontryagin Maximum Principle applied to the time-optimal problem $$
\dot x = u + W(x),\qquad u\in B_x,\qquad \text{minimize } \int_0^T 1\,dt.
$$
Pontryagin's Maximum Principle yields the pre-Hamiltonian
$$\mathcal{H}(x,\alpha,u) = \langle\alpha, u+W(x)\rangle - \xi_0,$$
 where $\alpha(t)\in T^*_xM$ and $\xi_0\geq 0$.
In the normal case $\xi_0>0$, we may normalize $\xi_0=1$ (this only reparameterizes time), the maximization over $u\in B_x\subset\mathcal{D}_x$
restricts the relevant covector to $\mathcal{D}^*_x$, which we henceforth denote $\xi\in\mathcal{D}^*_x$.
Define the degree-one maximized Hamiltonian
$$
h_Z(\xi)=\sup_{u\in B_x}\langle \xi,\,u+W(x)\rangle
=\langle \xi,W(x)\rangle+\sup_{u\in B_x}\langle \xi,u\rangle.
$$
Since $B_x$ is the $a$-unit ball in $\mathcal D_x$, the support function of $B_x$ is the dual norm
$$
\sup_{u\in B_x}\langle \xi,u\rangle=\|\xi\|_{a^*}=\sqrt{a^*(\xi,\xi)},
\qquad \xi\in\mathcal D_x^*.
$$ Hence
$$
h_Z(\xi)=\sqrt{a^*(\xi,\xi)}+\xi(W(x)).
$$
Passing to the energy Hamiltonian (the one generating the normal extremal flow) we set
\begin{equation}\label{eq:HZ}
H_Z(\xi)=\frac12\,h_Z(\xi)^2
=\frac12\big(\sqrt{a^*(\xi,\xi)}+\xi(W(x))\big)^2,
\qquad \xi\in\mathcal D_x^*.
\end{equation}
On the other hand, the dual norm of $F$ is the support function of the $F$-unit ball $W(x)+B_x$.
the $F$-unit ball is exactly
$$
\{v\in\mathcal D_x:\ F_x(v)\le 1\}=W(x)+B_x.
$$
Therefore its dual norm is
\begin{align*}
  F_x^*(\xi) &  =\sup\{\langle \xi,v\rangle:\ F_x(v)\le 1\}\\
   & =\sup_{v\in W(x)+B_x}\langle \xi,v\rangle  \\
   & = \sup_{u\in B_x}\langle \xi,W(x)+u\rangle \\
  & =\xi(W(x))+\sup_{u\in B_x}\langle \xi,u\rangle.
\end{align*}
Using again $\sup_{u\in B_x}\langle \xi,u\rangle=\sqrt{a^*(\xi,\xi)}$, we obtain
\begin{equation}\label{eq:Fdual}
F_x^*(\xi)=\sqrt{a^*(\xi,\xi)}+\xi(W(x)),\qquad \xi\in\mathcal D_x^*.
\end{equation}
The associated (energy) sub-Hamiltonian is
\begin{equation}\label{eq:HR}
H_R(\xi) =\frac12\big(F_x^*(\xi)\big)^2
=\frac12\big(\sqrt{a^*(\xi,\xi)}+\xi(W(x))\big)^2,
\qquad \xi\in\mathcal D_x^*.
\end{equation}
Comparing \eqref{eq:HZ} and \eqref{eq:HR} yields the desired identity
$$
H_Z(\xi)=H_R(\xi)\qquad\text{for all }\xi\in\mathcal D^*.
$$
Hence the normal Hamiltonian flows coincide, and their projections to $M$ coincide as curves
(up to orientation-preserving reparameterization). Combined with admissibility and length equivalence, this proves that
time-minimizing Zermelo trajectories are precisely the normal geodesics of $F$.
\end{proof}
The identification of $H_Z$ with $H_R$ simultaneously justifies the sub-Hamiltonian \eqref{H01}. Indeed, with $\beta_i = -a_{ij}W^j$, one has $\beta^i\xi_i = -\xi(W)$, so
$$
F^*(\xi) = \sqrt{a^{ij}\xi_i\xi_j} - \beta^i\xi_i, \qquad H(\xi) = \frac{1}{2}\bigl(\sqrt{a^{ij}\xi_i\xi_j}-\beta^i\xi_i\bigr)^2,
$$
which is precisely \eqref{H01}. The minus sign in \eqref{H01} is not a sign error: it encodes the convention $\beta^i = -W^i$, so that $-\beta^i\xi_i = +\xi(W)>0$ when $\xi$ and $W$ are co-oriented, consistently with $F^*>0$.
\subsection{Hopf-Rinow theorem for sub-Randers manifolds}

The sub-Randers metric $F$ is positively homogeneous of degree 1, which implies that the induced metric distance function $d$ is asymmetric, meaning that in general, $d(x, y) \neq d(y, x).$ Because $d$ is generally asymmetric, one must distinguish forward and backward completeness; the appropriate Hopf-Rinow-type equivalences are recalled below, see \cite{layth23, bao} for more details, one can also check \cite{Do92} for complete Riemannian manifolds.
The metric space $(M, d)$ is forward complete if every forward Cauchy sequence $\{x_n\}$ (i.e., $d(x_n, x_{n+1}) \to 0$) converges.
 Analogous to sub-Finsler geometry, the Hopf-Rinow theorem for sub-Randers manifolds is reformulated in terms of forward geodesic completeness, it means that every geodesic $\gamma(t)$, parametrized with constant sub-Randers speed $F(\dot\gamma(t))$ and initially defined on $t\in[0,1)$, extends indefinitely to $t\in[0,\infty)$.
 
 A set $K \subset M$ is forward bounded if $\sup_{y \in K} d(x, y) < \infty$ for some $x \in M$.

The {\it exponential map} $\exp_x: \mathcal{D}^*_x \to M$ is defined by
$\exp_x(t\xi) = \pi(\phi_t(x,\xi))$, where $\phi_t$ is the sub-Hamiltonian flow of
$H(\xi) = \tfrac{1}{2}\bigl(\sqrt{a^{ij}\xi_i\xi_j} - \beta^i\xi_i\bigr)^2$
on $\mathcal{D}^*\cong T^*M/\mathcal{D}^0$ and $\pi:\mathcal{D}^*\to M$ is the
canonical projection $\pi(x,\xi)=x$.

The sub-Hamiltonian's equations \eqref{HE} define a vector field $\vec{H} = \left( \frac{\partial H}{\partial \xi_i}, -\frac{\partial H}{\partial x^i} \right)$ on $\mathcal{D}^*$, more details are provided in \cite{layth19, layth23, layth24}. The sub-Hamiltonian flow $\phi_t: \mathcal{D}^* \to \mathcal{D}^*$ is the integral curve of $\vec{H}$
$$
\frac{d}{dt} \phi_t(x, \xi) = \vec{H}(\phi_t(x, \xi)), \quad \phi_0(x, \xi) = (x, \xi).
$$
Thus, $\phi_t(x, \xi) = (x(t), \xi(t))$ is the position and momentum at time $t$, starting from $(x, \xi)$ at $t = 0$. Explicitly $x(t)$ is the base point on $M$,
and $\xi(t)$ is the evolved covector in $\mathcal{D}^*_{x(t)}$.

\begin{thm}[Hopf-Rinow for Sub-Randers Manifolds]\label{thm:HopfRinow}
Let $(M, \mathcal{D}, F)$ be a connected sub-Randers manifold, with $\mathcal{D}$ bracket-generating and $F(v) = \sqrt{a(v, v)} + \beta(v)$, $\|\beta\|_a < 1$. The following assertions  are equivalent:
\begin{itemize}
  \item [I.] $(M, d)$ is forward complete as a metric space. 
  \item [II.] Every geodesic $\gamma(t)$, parametrized with constant \(F\)-speed, extends to \(t \in [0, \infty)\).  
  \item [III.] The exponential map is surjective for all $x \in M$, provided abnormal minimizers are not length-minimizing.  
  \item [IV.] Every closed, forward bounded subset of $M$ is compact.  
\end{itemize}
Furthermore, if any of the above statements holds, then for any $x, y \in M$, there exists a forward minimizing geodesic from $x$ to $y$. 
\end{thm}
\begin{proof}
Since $F(v) \geq (1 - \|\beta\|_a) \sqrt{a(v, v)}$ and $F(v) \leq (1 + \|\beta\|_a) \sqrt{a(v, v)}$, we have
$$
(1 - \|\beta\|_a) d_a(x, y) \leq d(x, y) \leq (1 + \|\beta\|_a) d_a(x, y),
$$
where $d_a$ is the sub-Riemannian distance. The bracket-generating condition ensures $d(x, y) < \infty$ (Chow-Rashevskii), \cite{Chow39}.

I $\Rightarrow$ II
Assume $(M, d)$ is forward complete. Let $\gamma: [0, T) \to M$ be a geodesic with $F(\dot{\gamma}) = s > 0$, maximal forward. If $T < \infty$, take $t_n < t_m \to T^-$
$$
d(\gamma(t_n), \gamma(t_m)) \leq \int_{t_n}^{t_m} F(\dot{\gamma}) \, dt = s (t_m - t_n) \to 0.
$$
The sequence $\{\gamma(t_n)\}$ is forward Cauchy
 $$d(\gamma(t_n), \gamma(t_{n+1})) \leq s (t_{n+1} - t_n) \to 0.$$ By completeness, $\gamma(t_n) \to x$. Locally near $x$,
the sub-Hamiltonian equations \eqref{HE}
extend $\gamma$ beyond $T$, a contradiction. Thus, $T = \infty$.

II $\Rightarrow$ III
Assume forward geodesic completeness. For $y \in M$, connectivity implies existence of a horizontal curve from $x$ to $y$. By PMP, normal geodesics satisfy $H > 0$, abnormals $H = 0$. If abnormals do not minimize length, then any minimizing geodesic $\gamma$ from $x$ to $y$ is normal. Completeness ensures $\exp_x(t \xi) = y$ for some $\xi \in \mathcal{D}^*_x$. Since $H$ is smooth and $\mathcal{D}^*_x$ is finite-dimensional, the flow covers $M$ (by bracket-generating), making $\exp$ surjective.

III $\Rightarrow$ IV
Assume the exponential map is surjective, with abnormal geodesics not length-minimizing. Let $K \subset M$ be closed and forward bounded, i.e., $d(x, y) < r$ for $y \in K$. For each $y \in K$, there exists $\xi_y \in \mathcal{D}^*_x$ with $\exp_x(t_y \xi_y) = y$, where $t_y F(\xi_y) = d(x, y) < r$. Since $F(\xi_y) \geq (1 - \|\beta\|_a) \sqrt{a^*(\xi_y, \xi_y)}$, $\{\xi_y\}$ is bounded. The set $\{\xi \in \mathcal{D}^*_x \mid F(\xi) < r/(1 - \|\beta\|_a)\}$ is compact, and $\exp_x$ continuous, so $K \subseteq \exp_x(\text{bounded set})$ is compact as a closed subset of a compact set.

IV $\Rightarrow$ I
Assume closed, forward bounded sets are compact. Let $\{x_n\}$ be forward Cauchy with $\sum d(x_n, x_{n+1}) < \infty.$ Define $s_n = \sum_{k=n}^\infty d(x_k, x_{k+1})$; then $d(x_n, x_m) \leq s_n - s_m \to 0$ as $n < m \to \infty$. The set $K = \{x_n\}$ is forward bounded $$d(x_1, x_n) \leq \sum_{k=1}^{n-1} d(x_k, x_{k+1}) < \infty.$$ Since $\overline{K}$ is compact, $\{x_n\}$ has a convergent subsequence $x_{n_k} \to x$. For $\epsilon > 0$, choose $N$ so $s_n - s_m < \epsilon/2$ for $n, m > N$, and $d(x_{n_k}, x) < \epsilon/2$ for some $n_k > N$. Then
$$
d(x_n, x) \leq d(x_n, x_{n_k}) + d(x_{n_k}, x) < \epsilon,
$$
so $x_n \to x$. Thus, $(M, d)$ is forward complete.

Under condition  \textup{(IV)}, for $x, y \in M$, let $\Phi = \{\gamma \mid \gamma(0) = x, \gamma(1) = y, \text{horizontal}\}$. The set $K = \{z \mid d(x, z) \leq d(x, y) + 1\}$ is compact. A sequence $\{\gamma_n\} \subset \Phi$ with $\ell(\gamma_n) \to d(x, y)$ is equicontinuous since $$d(\gamma_n(t), \gamma_n(s)) \leq |t - s| (d(x, y) + 1).$$ By Arzel\`a-Ascoli theorem, $\gamma_{n_k} \to \gamma$ uniformly.
 Since $\ell$ is lower semicontinuous,
  $$\ell(\gamma)\leq\liminf_{k}\ell(\gamma_{n_k}) = d(x,y).$$
  By definition of $d$, $\ell(\gamma)\geq d(x,y)$, so $\ell(\gamma)=d(x,y)$ and $\gamma$ is a minimizing geodesic. If abnormal geodesics are not length-minimizing, then $\gamma$ is normal, hence $\gamma = \exp_x(t\xi)$ for some $\xi\in\mathcal{D}^*_x$. \end{proof}
\bigskip

\noindent
{\bf Acknowledgment:} We sincerely thank the Editors and the Reviewers for their time, effort, and constructive comments, which significantly improved the quality and clarity of this work.

\providecommand{\bysame}{\leavevmode\hbox
to3em{\hrulefill}\thinspace}

\end{document}